\documentclass[a4paper,oneside]{amsart}

\usepackage{amssymb,amsthm}
\usepackage{hyperref}

\newtheorem{thm}{Theorem}[section]
\newtheorem{pro}[thm]{Proposition}
\newtheorem{lem}[thm]{Lemma}

\theoremstyle{definition}

\newtheorem{rem}[thm]{Remark}

\newcommand{\en}{\mathbb{N}}

\newcommand{\er}{\mathbb{R}}

\newcommand{\ce}{\mathbb{C}}

\newcommand{\me}{\mathrm{e}}

\begin{document}

\title[Coarse and uniform embeddability]{On the equivalence between coarse and uniform embeddability of quasi-Banach spaces into a Hilbert space}

\author{Michal Kraus}

\address{Pohang Mathematics Institute, Pohang University of Science and Technology, San 31 Hyoja Dong, Nam-Gu, Pohang 790-784, Republic of Korea}

\email{mkraus@karlin.mff.cuni.cz or mkraus@postech.ac.kr}

\subjclass[2010]{Primary 46B20; Secondary 51F99}

\keywords{Coarse embedding, uniform embedding, quasi-Banach space, Hilbert space}

\begin{abstract}
We give a direct proof of the fact that a quasi-Banach space coarsely embeds into a Hilbert space if and only if it uniformly embeds into a Hilbert space.
\end{abstract}

\thanks{This work was supported by Priority Research Centers Program through the National Research Foundation of Korea (NRF) funded by the Ministry of Education, Science and Technology (Project No. 2012047640), and by the grants GA\v{C}R 201/11/0345 and PHC Barrande 2012-26516YG}


\maketitle

\section{Introduction}

Let $(M,d_M),(N,d_N)$ be metric spaces and let $T:M\to N$ be a mapping. Then $T$ is called a \emph{coarse embedding} if there exist nondecreasing functions $\rho_1,\rho_2:[0,\infty)\to[0,\infty)$ such that $\lim_{t\to\infty}\rho_1(t)=\infty$ and
$$\rho_1(d_M(x,y))\leq d_N(T(x),T(y))\leq\rho_2(d_M(x,y))\ \text{for all }x,y\in M.$$
We say that $T$ is a \emph{uniform embedding} if $T$ is injective and both $T$ and $T^{-1}:T(M)\to M$ are uniformly continuous. If $T$ is both a coarse embedding and a uniform embedding, then we call it a \emph{strong uniform embedding}. We say that $M$ \emph{coarsely embeds} into $N$ if there exists a coarse embedding of $M$ into $N$, and similarly for other types of embeddings. The reader should be warned that what we call a coarse embedding is called a uniform embedding by some authors. We use the term coarse embedding because in the nonlinear geometry of Banach spaces the term uniform embedding already has a well-established meaning as above. Let us mention that all vector spaces in this paper are supposed to be real.

Aharoni, Maurey and Mityagin \cite[Theorem~4.1]{amm} proved that a linear metric space uniformly embeds into a Hilbert space if and only if it is linearly isomorphic to a subspace of $L_0(\mu)$ for some probability space $(\Omega,\mathcal{B},\mu)$ ($L_0(\mu)$ is the space of all (equivalence classes of) measurable functions on $(\Omega,\mathcal{B},\mu)$ with the topology of convergence in probability). Later, Randrianarivony \cite[Theorem~1]{ra} proved by using similar reasoning that a quasi-Banach space coarsely embeds into a Hilbert space if and only if it is linearly isomorphic to a subspace of $L_0(\mu)$ for some probability space $(\Omega,\mathcal{B},\mu)$. Her proof actually provides a strong uniform embedding whenever the condition on the right hand side is satisfied. As a consequence, a quasi-Banach space coarsely embeds into a Hilbert space if and only if it uniformly embeds into a Hilbert space. This is quite surprising, because a coarse embedding, by definition, gives information on large distances, while a uniform embedding gives information on small distances. In summary, we have the following theorem.

\begin{thm}\label{main_thm}
Let $X$ be a quasi-Banach space. The following assertions are equivalent.
\renewcommand{\labelenumi}{\rm(\roman{enumi})}
\begin{enumerate}
\item $X$ coarsely embeds into a Hilbert space.
\item $X$ uniformly embeds into a Hilbert space.
\item $X$ strongly uniformly embeds into a Hilbert space.
\item There exists a probability space $(\Omega,\mathcal{B},\mu)$ such that $X$ is linearly isomorphic to a subspace of $L_0(\mu)$.
\end{enumerate}
\end{thm}

The purpose of this paper is to give a direct proof of the equivalence of the conditions (i), (ii) and (iii), without passing through the condition (iv). The proof that (i) implies (iii) is a refinement of (a part of) an argument of Johnson and Randrianarivony given in \cite{jr}, where the authors prove that $\ell_p$ for $p>2$ does not coarsely embed into a Hilbert space. The proof that (ii) implies (i) is a simple application of a (slightly improved) characterization of coarse embeddability of metric spaces into a Hilbert space due to Dadarlat and Guentner \cite[Proposition~2.1]{dg}. For the sake of completeness, we give proofs even of some known results which are used in our proof.

\section{Preliminaries}

Let $X$ be a quasi-Banach space. By a theorem of Aoki and Rolewicz (see e.g. \cite[Proposition~H.2]{bl}), there exists an equivalent quasi-norm $\|.\|$ on $X$ and $0<p\leq1$ such that $\|x+y\|^p\leq\|x\|^p+\|y\|^p$ for all $x,y\in X$. In particular, $(x,y)\mapsto\|x-y\|^p$ is an invariant metric on $X$ determining the topology given by the original quasi-norm. All metric properties of the space $X$ will be regarded with respect to this metric. It is easy to see that if $d_1(x,y)=\|x-y\|_1^p$ and $d_2(x,y)=\|x-y\|_2^q$ are two such metrics on $X$, $M$ is a metric space and $T:X\to M$ is a mapping, then $T:(X,d_1)\to M$ is a coarse (uniform) embedding if and only if $T:(X,d_2)\to M$ is. So for our purposes we may pick any such metric on $X$. For a brief overview of quasi-Banach spaces see for example \cite[Appendix~H]{bl}.

Important tools for studying coarse and uniform embeddings into Hilbert spaces are positive and negative definite kernels and functions. Let us recall the definitions and basic facts used in the sequel. For a detailed exposition see for example \cite[Chapter~8]{bl}.

A kernel $K$ on a set $X$ (i.e. a function $K:X\times X\to\ce$ such that $K(y,x)=\overline{K(x,y)}$ for all $x,y\in X$) is called 
\begin{enumerate}\renewcommand{\labelenumi}{\rm(\alph{enumi})}
\item \emph{positive definite} if $\sum_{i,j=1}^nK(x_i,x_j)c_i\overline{c_j}\geq0$ for all $n\in\en$, $x_1,\dots,x_n\in X$ and $c_1,\dots,c_n\in\ce$,
\item \emph{negative definite} if $\sum_{i,j=1}^nK(x_i,x_j)c_i\overline{c_j}\leq0$ for all $n\in\en$, $x_1,\dots,x_n\in X$ and $c_1,\dots,c_n\in\ce$ satisfying $\sum_{i=1}^n c_i=0$. 
\end{enumerate}
If $X$ is an abelian group and $f:X\to\ce$ is a function, then we say that $f$ is \emph{positive (negative) definite} if $(x,y)\mapsto f(x-y)$ is a positive (negative) definite kernel on $X$.

Note that if the kernel $K$ is real-valued, then in order to check the positive or negative definiteness of $K$ it suffices to use only the real scalars. In the sequel, we will work only with real-valued positive (negative) definite kernels and functions.

There is a relation between positive and negative definite kernels as described by the following theorem of Schoenberg (for a proof see e.g. \cite[Proposition 8.4]{bl}).

\begin{thm}\label{neg_pos_relation}
A kernel $N$ on a set $X$ is negative definite if and only if $\me^{-tN}$ is positive definite for every $t>0$.
\end{thm}

The key result for our purposes is the following theorem. Part (i) was probably first proved by Moore, part (ii) is due to Schoenberg. For a proof see for example \cite[Proposition~8.5]{bl}.

\begin{thm}\label{moore-schoe}
Let $X$ be a set.
\renewcommand{\labelenumi}{\rm(\alph{enumi})}
\begin{enumerate}
\item A kernel $K$ on $X$ is positive definite if and only if there exists a Hilbert space $H$ and a mapping $T:X\to H$ such that $$K(x,y)=\left\langle T(x),T(y)\right\rangle\ \text{for all }x,y\in X.$$
\item A real-valued kernel $N$ on $X$ satisfying $N(x,x)=0$ for every $x\in X$ is negative definite if and only if there exists a Hilbert space $H$ and a mapping $T:X\to H$ such that 
$$N(x,y)=\|T(x)-T(y)\|^2\ \text{for all }x,y\in X.$$
\end{enumerate}
\end{thm}

The proof of the necessity in both (a) and (b) given in \cite[Proposition~8.5]{bl} actually leads to a complex Hilbert space $H$, but it is easy to see that if the kernel $K$ in (a) is real-valued, then there exists a real Hilbert space $H$ with the desired properties. In (b) we can always find a real Hilbert space.

A simple consequence of Theorem \ref{moore-schoe} is that if $X$ is an abelian group and $f:X\to\er$ is a positive definite function with $f(0)=1$, then $|f(x)|\leq1$ for every $x\in X$. Similarly, if $f:X\to\er$ is a negative definite function with $f(0)=0$, then $f(x)\geq0$ for every $x\in X$.

We will also need the following fact (\cite[p. 186, Examples. (iii)]{bl}).

\begin{lem}\label{neg_def_pow}
Let $N$ be a negative definite kernel on a set $X$ such that $N(x,y)\geq0$ for all $x,y\in X$ and let $0<a<1$. Then $N^a$ is also a negative definite kernel. 
\end{lem}

Let $X$ be an abelian group with an invariant metric $d$ and let $f:X\to\er$ be a positive definite function with $f(0)=1$. Then Theorem \ref{moore-schoe}(a) yields a Hilbert space $H$ and a mapping $T:X\to H$ such that 
$$f(x-y)=\left\langle T(x),T(y)\right\rangle\ \text{for all }x,y\in X.$$
Then $\|T(x)\|=1$ for every $x\in X$ and 
$$\|T(x)-T(y)\|^2=2(1-f(x-y))\ \text{for all }x,y\in X.$$
This reasoning leads easily to the following result (\cite[Proposition 3.2]{amm}). First, denote
\begin{equation}\label{def_g_f}
g_f(t)=\inf\{1-f(x):d(x,0)\geq t\},\ t>0.
\end{equation}

\begin{pro}\label{amm_group_emb_into_sphere}
Let $X$ be an abelian group with an invariant metric. If there exists a continuous positive definite function $f:X\to\er$ such that $f(0)=1$ and $g_f(t)>0$ for every $t>0$, then $X$ uniformly embeds into the unit sphere of a Hilbert space. 
\end{pro}

The hypothesis of Proposition \ref{amm_group_emb_into_sphere} actually characterizes metric abelian groups uniformly embeddable into a Hilbert space, see \cite[Theorem~3.1]{amm}.

Now, let $H$ be a Hilbert space. By Theorem \ref{moore-schoe}(b), the function $\|.\|^2$ on $H$ is negative definite, and therefore, by Theorem \ref{neg_pos_relation}, the function $f(x)=\me^{-\|x\|^2},x\in H$, is positive definite. It is clear that $f$ satisfies the hypothesis of Proposition \ref{amm_group_emb_into_sphere}, and therefore $H$ uniformly embeds into the unit sphere of a Hilbert space. As a consequence, we obtain the following fact (\cite[paragraph after Corollary 3.3]{amm}).

\begin{pro}\label{red_sphere}
If a metric space uniformly embeds into a Hilbert space, then it uniformly embeds into the unit sphere of a Hilbert space.
\end{pro}

Finally, let us introduce two moduli which will be useful. Let $(M,d_M),(N,d_N)$ be metric spaces and let $T:M\to N$ be a mapping. For $t>0$ define
\begin{equation}\label{low_mod}
\varphi_T(t)=\inf\{d_N(T(x),T(y)):d_M(x,y)\geq t\}
\end{equation}
and
\begin{equation}\label{up_mod}
\omega_T(t)=\sup\{d_N(T(x),T(y)):d_M(x,y)\leq t\}.
\end{equation}
Then the functions $\varphi_T$ and $\omega_T$ are nondecreasing and
$$\varphi_T(d_M(x,y))\leq d_N(T(x),T(y))\leq\omega_T(d_M(x,y))\ \text{for all }x,y\in M,x\neq y.$$
It is easy to see that $T$ is a coarse embedding if and only if $\lim_{t\to\infty}\varphi_T(t)=\infty$ and $\omega_T(t)<\infty$ for every $t>0$, and $T$ is a uniform embedding if and only if $\varphi_T(t)>0$ for every $t>0$ and $\lim_{t\to0}\omega_T(t)=0$. Some authors use this as a definition of a coarse and uniform embedding.

\section{Proofs}

Let $X$ be a quasi-Banach space. We will suppose from now on that the norm on $X$ satisfies $\|x+y\|^p\leq\|x\|^p+\|y\|^p$ for all $x,y\in X$, where $0<p\leq1$, and we will work with the metric $(x,y)\mapsto\|x-y\|^p$ on $X$.

Let us first prove that (i) implies (iii) in Theorem \ref{main_thm}. We will use the following property of negative definite functions (Proposition \ref{rho_control}).

\begin{lem}\label{neg_def_growth}
Let $f:X\to\er$ be a negative definite function with $f(0)=0$. Then for every $x\in X$ and $n\in\en$ we have $$f(nx)\leq n^2f(x).$$
\end{lem}
\begin{proof}
By Theorem \ref{moore-schoe}(b), there exists a Hilbert space $H$ and a mapping $T:X\to H$ such that $f(x-y)=\|T(x)-T(y)\|^2$ for all $x,y\in X$. Let $x,y\in X$. Then
\begin{align*}
\sqrt{f(x+y)}&=\|T(x)-T(-y)\|\leq\|T(x)-T(0)\|+\|T(0)-T(-y)\|\\
&=\sqrt{f(x)}+\sqrt{f(y)}
\end{align*}
Hence for every $x\in X$ and $n\in\en$ we have $\sqrt{f(nx)}\leq n\sqrt{f(x)}$.
\end{proof}

If $f:X\to\er$ is a negative definite function with $f(0)=0$, we define 
$$\rho_f(t)=\inf\{f(x):\|x\|^p\geq t\},\ t>0.$$
Note that $\rho_f(t)\geq0$ for every $t>0$ since $f(x)\geq0$ for every $x\in X$.

\begin{pro}\label{rho_control}
Let $f:X\to\er$ be a negative definite function with $f(0)=0$. If $\rho_f(t)>0$ for some $t>0$, then $\rho_f(t)>0$ for every $t>0$.
\end{pro}
\begin{proof}
Suppose that there exists $t>0$ such that $\rho_f(t)=0$ and let $s>t$. Let $\varepsilon>0$. Then there exists $x\in X$ such that $\|x\|^p\geq t$ and $f(x)<\varepsilon$. Let $n\in\en$ be such that $\|(n-1)x\|^p<s\leq\|nx\|^p$. Then $n<2\left(\frac{s}{t}\right)^\frac{1}{p}$. Hence, by Lemma \ref{neg_def_growth},
$$\rho_f(s)\leq f(nx)\leq n^2f(x)<4\left(\frac{s}{t}\right)^\frac{2}{p}\varepsilon.$$
Since $\varepsilon>0$ was arbitrary, we have $\rho_f(s)=0$. Since $\rho_f$ is nondecreasing, we see that $\rho_f(t)=0$ for every $t>0$. 
\end{proof}

\begin{rem}\label{amm_ineq}
By a slight modification of the proof of Proposition \ref{rho_control} (replacing $\varepsilon$ by $\rho_f(t)+\varepsilon$), we can actually prove that if $f:X\to\er$ is a negative definite function with $f(0)=0$ and $0<t<s$, then 
$$\rho_f(s)\leq\frac{4\rho_f(t)}{t^\frac{2}{p}}s^\frac{2}{p}.$$
If $f:X\to\er$ is a positive definite function with $f(0)=1$, then $1-f$ is negative definite and $(1-f)(0)=0$. Since $g_f=\rho_{1-f}$ (recall that $g_f$ was defined in \eqref{def_g_f}), we obtain, for $0<t<s$,
\begin{equation}\label{amm_est}
g_f(s)\leq\frac{4g_f(t)}{t^\frac{2}{p}}s^\frac{2}{p}.
\end{equation}
This was proved in \cite[Corollary~4.9]{amm} in the case when $X$ is a normed linear space (with $p=1$).
\end{rem}

\begin{proof}[Proof of Theorem \ref{main_thm}, (i)$\Rightarrow$(iii)]
Let $T:X\to H$ be a coarse embedding, where $H$ is a Hilbert space, and let $\rho_1,\rho_2:[0,\infty)\to[0,\infty)$ be nondecreasing functions satisfying $\lim_{t\to\infty}\rho_1(t)=\infty$ and 
$$\rho_1(\|x-y\|^p)\leq\|T(x)-T(y)\|\leq\rho_2(\|x-y\|^p)\ \text{for all }x,y\in X.$$ 
We proceed in five steps. The first three steps are essentially Steps 0, 1, 2 and a part of Step 3 from \cite{jr}, which extend to the case of quasi-Banach spaces (cf. also \cite[Proposition 2]{ra}).

\textbf{Step 1:} In the first step we show that we may assume without loss of generality that $\rho_2(t)=t^a$ for some $a>0$.

First, we may assume that 
\begin{equation}\label{coa_Lip_red}
\|T(x)-T(y)\|\leq\|x-y\|\ \text{if }\|x-y\|^p\geq1.
\end{equation}
Indeed, let $\|x-y\|^p\geq1$, hence $\|x-y\|\geq1$. Let $n\in\en$ satisfy $n-1<\|x-y\|\leq n$. Then $n<2\|x-y\|$. We may find $x=x_0,x_1,\dots,x_n=y$ in X such that $\|x_i-x_{i-1}\|\leq1$, hence $\|x_i-x_{i-1}\|^p\leq1$, for $i=1,\dots,n$ (set $x_i=x+i\frac{y-x}{n},i=0,\dots,n$). Then 
\begin{align*}
\|T(x)-T(y)\|&=\left\|\sum_{i=1}^n\left(T(x_i)-T(x_{i-1})\right)\right\|\leq\sum_{i=1}^n\|T(x_i)-T(x_{i-1})\|\\
&\leq\sum_{i=1}^n\rho_2(\|x_i-x_{i-1}\|^p)\leq n\rho_2(1)<2\rho_2(1)\|x-y\|.
\end{align*}
So, by rescaling, we may indeed assume that \eqref{coa_Lip_red} holds.

Now, $(x,y)\mapsto\|T(x)-T(y)\|^2$ is a negative definite kernel on $X$ by Theorem~\ref{moore-schoe}(b). Let $0<r\leq1$. By Lemma~\ref{neg_def_pow}, the kernel $N(x,y)=\|T(x)-T(y)\|^{2r}$ on $X$ is also negative definite and satisfies $N(x,x)=0$ for every $x\in X$. By Theorem~\ref{moore-schoe}(b), there exists a Hilbert space $H_r$ and a mapping $T_r:X\to H_r$ such that $N(x,y)=\|T_r(x)-T_r(y)\|^2$ for all $x,y\in X$. Hence, by \eqref{coa_Lip_red}, $$\|T_r(x)-T_r(y)\|=\|T(x)-T(y)\|^r\leq(\|x-y\|^p)^\frac{r}{p}\ \text{if }\|x-y\|^p\geq1.$$

Let $\mathcal{N}$ be a 1-net in $X$ (i.e. a maximal subset of $X$ such that all pairs of its distinct points have distance at least 1). Let $\frac{r}{p}\leq\frac{1}{2}$. By \cite[last statement of Theorem~19.1]{ww}, the restriction of $T_r$ to the set $\mathcal{N}$ can be extended to a mapping $\widetilde{T_r}:X\to H_r$ satisfying 
$$\|\widetilde{T_r}(x)-\widetilde{T_r}(y)\|\leq(\|x-y\|^p)^\frac{r}{p}\ \text{for all }x,y\in X.$$ 
It is easy to see that $\varphi_{\widetilde{T_r}}(t)\to\infty$ as $t\to\infty$ ($\varphi_{\widetilde{T_r}}$ was defined in \eqref{low_mod}), and therefore $\widetilde{T_r}$ is a coarse embedding.

So from now on we will assume that the coarse embedding $T:X\to H$ satisfies
$$\rho_1(\|x-y\|^p)\leq\|T(x)-T(y)\|\leq(\|x-y\|^p)^a\ \text{for all }x,y\in X,$$
where $\rho_1:[0,\infty)\to[0,\infty)$ is a nondecreasing function satisfying $\lim_{t\to\infty}\rho_1(t)=\infty$ and $a>0$.

\textbf{Step 2:} Define $N(x,y)=\|T(x)-T(y)\|^2,x,y\in X$. By Theorem~\ref{moore-schoe}(b), N is a negative definite kernel on $X$. Define $\varphi(t)=(\rho_1(t))^2,t\geq0$. Then 
$$\varphi(\|x-y\|^p)\leq N(x,y)\leq(\|x-y\|^p)^{2a}\ \text{for all }x,y\in X,$$
and $\varphi:[0,\infty)\to[0,\infty)$ is nondecreasing and $\lim_{t\to\infty}\varphi(t)=\infty$.

\textbf{Step 3:} In this step we obtain a continuous negative definite function $f:X\to\er$ such that
$$\varphi(\|x\|^p)\leq f(x)\leq(\|x\|^p)^{2a}\ \text{for every }x\in X.$$
(In particular $f(0)=0$.)

Let $M:\ell_\infty(X)\to\er$ be an invariant mean, i.e. $M$ is linear and
\begin{enumerate}
\item $M(1)=1$.
\item $M(g)\geq0$ for every $g\geq0$.
\item $M(g_x)=M(g)$ for all $g\in\ell_\infty(X)$ and $x\in X$, where $g_x(y)=g(y+x),y\in X$.
\end{enumerate}
The existence of such a functional is ensured for example by \cite[Theorem C.1]{bl}. If $x,y\in X$, define $N_{x,y}(z)=N(z+x,z+y),z\in X$. Let
$$f(x)=M(N_{x,0}),\ x\in X.$$
Let us show that $f$ is the desired function on $X$.

First, $f$ is well-defined, since $|N_{x,0}(y)|=|N(y+x,y)|\leq(\|x\|^p)^{2a}$ for every $y\in X$ and therefore $N_{x,0}\in\ell_\infty(X)$.

To show that $f$ is negative definite, let us first show that $(x,y)\mapsto f(x-y)$ is indeed a kernel, i.e. $f(-x)=f(x)$ for every $x\in X$. Let $x\in X$. Then $N_{-x,0}=N_{0,-x}$ and by the translation invariance of $M$ (condition (3)) we have
$$f(-x)=M(N_{-x,0})=M(N_{0,-x})=M(N_{x,0})=f(x).$$
Now, let $x_1,\dots,x_n\in X$ and let $c_1,\dots,c_n\in\er$ satisfy $\sum_{i=1}^nc_i=0$. Then
\begin{align*}
\sum_{i,j=1}^nf(x_i-x_j)c_ic_j&=\sum_{i,j=1}^nM(N_{x_i-x_j,0})c_ic_j=\sum_{i,j=1}^nM(N_{x_i,x_j})c_ic_j\\
&=M\left(\sum_{i,j=1}^nN_{x_i,x_j}c_ic_j\right)\leq M(0)=0,
\end{align*}
where the second equality follows from the translation invariance of $M$ and the inequality from the positivity of $M$ (condition (2)) and the negative definiteness of $N$.

Let us now show that $f$ is continuous. Let $x,y\in X$. First, if $z\in X$, then
\begin{align*}
|N_{x,0}(z)-N_{y,0}(z)|&=|N(z+x,z)-N(z+y,z)|\\
&=\left|\|T(z+x)-T(z)\|^2-\|T(z+y)-T(z)\|^2\right|\\
&=\left(\|T(z+x)-T(z)\|+\|T(z+y)-T(z)\|\right)\\
&\quad\cdot\left|\|T(z+x)-T(z)\|-\|T(z+y)-T(z)\|\right|\\
&\leq\left(\|T(z+x)-T(z)\|+\|T(z+y)-T(z)\|\right)\\
&\quad\cdot\|T(z+x)-T(z+y)\|\\
&\leq\left((\|x\|^p)^a+(\|y\|^p)^a\right)(\|x-y\|^p)^a
\end{align*}
Hence, by the positivity of $M$ and the fact that $M(1)=1$,
\begin{align*}
|f(x)-f(y)|&=|M(N_{x,0})-M(N_{y,0})|=|M(N_{x,0}-N_{y,0})|\\
&\leq M(|N_{x,0}-N_{y,0}|)\\
&\leq M\left(\left((\|x\|^p)^a+(\|y\|^p)^a\right)(\|x-y\|^p)^a\right)\\
&=\left((\|x\|^p)^a+(\|y\|^p)^a\right)(\|x-y\|^p)^a,
\end{align*}
and therefore $f$ is continuous.

Finally, let $x\in X$. Since
$$\varphi(\|x\|^p)\leq N_{x,0}(y)\leq(\|x\|^p)^{2a}$$
for every $y\in X$, we have
$$\varphi(\|x\|^p)\leq M(N_{x,0})\leq(\|x\|^p)^{2a}$$
and therefore
$$\varphi(\|x\|^p)\leq f(x)\leq(\|x\|^p)^{2a}.$$

So now we have a continuous negative definite function $f:X\to\er$ such that
$$\varphi(\|x\|^p)\leq f(x)\leq(\|x\|^p)^{2a}\ \text{for every }x\in X,$$
where $\varphi:[0,\infty)\to[0,\infty)$ is a nondecreasing function satisfying $\lim_{t\to\infty}\varphi(t)=\infty$ and $a>0$.

\textbf{Step 4:} We may assume that $\varphi(t)>0$ for every $t>0$. Indeed, since $\rho_f(t)\geq\varphi(t)$ for every $t>0$ and $\varphi(t)\to\infty$ as $t\to\infty$, we have $\rho_f(t)>0$ for some $t>0$. By Proposition \ref{rho_control}, $\rho_f(t)>0$ for every $t>0$. So we may set $\varphi(t)=\rho_f(t)$ for $t>0$.

\textbf{Step 5:} By Theorem~\ref{moore-schoe}(b), there exists a Hilbert space $H'$ and a mapping $S:X\to H'$ such that $f(x-y)=\|S(x)-S(y)\|^2$ for all $x,y\in X$. Hence
$$\sqrt{\varphi(\|x-y\|^p)}\leq\|S(x)-S(y)\|\leq(\|x-y\|^p)^a\ \text{for all }x,y\in X.$$
Since $\varphi(t)>0$ for every $t>0$ and $\lim_{t\to\infty}\varphi(t)=\infty$, we see that $S$ is a strong uniform embedding.
\end{proof}

\begin{rem}
If we want to prove that (i) implies (ii) in Theorem~\ref{main_thm} and do not mind that the resulting uniform embedding is not strong uniform, we may replace Steps 4 and 5 in the above proof by the following.

By Theorem \ref{neg_pos_relation}, the function $h=\me^{-f}$ is continuous positive definite, and satisfies $h(0)=1$ and 
$$h(x)\leq\me^{-\varphi(\|x\|^p)}\ \text{for every }x\in X.$$
Since $g_h(t)\geq1-\me^{-\varphi(t)}$ for every $t>0$ ($g_h$ was defined in \eqref{def_g_f}) and $\varphi(t)\to\infty$ as $t\to\infty$, we have $g_h(t)>0$ for some $t>0$. Since $1-h$ is negative definite, $(1-h)(0)=0$ and $\rho_{1-h}=g_h$, it follows from Proposition~\ref{rho_control} that $g_h(t)>0$ for every $t>0$ (this of course follows also from \eqref{amm_est}). Hence, by Proposition~\ref{amm_group_emb_into_sphere}, $X$ uniformly embeds into a Hilbert space.
\end{rem}

It is clear that (iii) implies (ii) in Theorem \ref{main_thm}. Let us turn to the proof that (ii) implies (i). We will need the following result of Dadarlat and Guentner essentially contained in the proof of \cite[Proposition~2.1]{dg} (cf. also \cite[Theorem 3]{no}).

\begin{pro}\label{dg_prop_2.1}
Let $M$ be a metric space. Suppose that there exists a $\delta>0$ such that for every $R>0$ and $\varepsilon>0$ there exists a Hilbert space H and a mapping $T:M\to S_H$ (where $S_H$ stands for the unit sphere of $H$) satisfying $\omega_T(R)\leq\varepsilon$ and $\lim_{t\to\infty}\varphi_T(t)\geq\delta$. (Recall that $\varphi_T$ and $\omega_T$ were defined in \eqref{low_mod} and \eqref{up_mod} respectively.) Then $M$ coarsely embeds into a Hilbert space.
\end{pro}

\begin{rem}
The converse holds too, see \cite[Proposition~2.1]{dg}.
\end{rem}

\begin{proof}
Denote the metric on $M$ by $d$. By the assumption, for every $n\in\en$ there exist a Hilbert space $H_n$, a mapping $T_n:M\to S_{H_n}$ and $s_n>0$ such that $\omega_{T_n}(\sqrt{n})\leq\frac{1}{2^n}$ and $\varphi_{T_n}(s_n)\geq\frac{\delta}{2}$. We may suppose that $s_1<s_2<...$ and that $s_n\to\infty$. Let $s_0=0$.

Choose an arbitrary $x_0\in M$ and define a mapping $T$ from $M$ to the Hilbert space $\left(\sum_{n=1}^\infty H_n\right)_{\ell_2}$ (the $\ell_2$-sum of the spaces $H_n$) by $$T(x)=\left(T_n(x)-T_n(x_0)\right)_{n=1}^\infty$$ 
(the fact that $T(x)\in\left(\sum_{n=1}^\infty H_n\right)_{\ell_2}$ for every $x\in M$ follows from the first estimate below). Let us show that $T$ is a coarse embedding. 

Let $x,y\in M$. Let $N\in\en$ be such that $\sqrt{N-1}\leq d(x,y)<\sqrt{N}$. Then
\begin{align*}
\|T(x)-T(y)\|^2&=\sum_{n=1}^{N-1}\|T_n(x)-T_n(y)\|^2+\sum_{n=N}^\infty\|T_n(x)-T_n(y)\|^2\\
&\leq\sum_{n=1}^{N-1}4+\sum_{n=N}^\infty\left(\omega_{T_n}(\sqrt{n})\right)^2\leq4(N-1)+\sum_{n=N}^\infty\frac{1}{4^n}\\
&\leq4(d(x,y))^2+\frac{1}{3}.
\end{align*}
On the other hand, if $N\in\en$ is such that $s_{N-1}\leq d(x,y)<s_N$, then
$$\|T(x)-T(y)\|^2\geq\sum_{n=1}^{N-1}\|T_n(x)-T_n(y)\|^2\geq\sum_{n=1}^{N-1}\left(\varphi_{T_n}(s_n)\right)^2\geq\left(\frac{\delta}{2}\right)^2(N-1).$$

Now, define functions $\rho_1,\rho_2:[0,\infty)\to[0,\infty)$ by $$\rho_1(t)=\sum_{n=1}^\infty\frac{\delta}{2}\sqrt{n-1}\chi_{[s_{n-1},s_n)}(t)$$ 
(where $\chi_{[s_{n-1},s_n)}$ is a characteristic function of the set $[s_{n-1},s_n)$) and 
$$\rho_2(t)=\sqrt{4t^2+\frac{1}{3}}.$$ 
Then $\rho_1,\rho_2$ are nondecreasing, $\lim_{t\to\infty}\rho_1(t)=\infty$ and for all $x,y\in M$ we have $$\rho_1(d(x,y))\leq\|T(x)-T(y)\|\leq\rho_2(d(x,y)).$$ 
Hence $T$ is a coarse embedding.
\end{proof}

\begin{proof}[Proof of Theorem \ref{main_thm}, (ii)$\Rightarrow$(i)]
Suppose that $X$ uniformly embeds into a Hilbert space. Let us show that $X$ satisfies the hypothesis of Proposition \ref{dg_prop_2.1}. 

By Proposition \ref{red_sphere}, there exists a Hilbert space $H$ and a uniform embedding $T:X\to S_H$. Then $\varphi_T(t)>0$ for every $t>0$ and $\lim_{t\to0}\omega_T(t)=0$. 

Let $a>0$. Define $T_a(x)=T(ax),x\in X$. Then $T_a$ maps $X$ into $S_H$. Since 
$$\varphi_T(\|x-y\|^p)\leq\|T(x)-T(y)\|\leq\omega_T(\|x-y\|^p)$$
for all $x,y\in X,x\neq y$, we have
$$\varphi_T(a^p\|x-y\|^p)\leq\|T_a(x)-T_a(y)\|\leq\omega_T(a^p\|x-y\|^p)$$
for all $x,y\in X,x\neq y$. Hence, for every $t>0$,
$$\varphi_{T_a}(t)\geq\varphi_T(a^pt)$$
and
$$\omega_{T_a}(t)\leq\omega_T(a^pt).$$

Now, let $R>0$ and $\varepsilon>0$. Then 
$$\omega_{T_a}(R)\leq\omega_T(a^pR),$$ 
and since $\omega_T(t)\to0$ as $t\to0$, we may choose $a>0$ so that $\omega_{T_a}(R)\leq\varepsilon$. On the other hand, $$\lim_{t\to\infty}\varphi_{T_a}(t)\geq\lim_{t\to\infty}\varphi_T(a^pt)=\lim_{t\to\infty}\varphi_T(t)$$ 
and the last limit does not depend on $a$ and it is positive since $\varphi_T$ is nondecreasing and $\varphi_T(t)>0$ for every $t>0$. So if we define $\delta=\lim_{t\to\infty}\varphi_T(t)$, we see that the hypothesis of Proposition \ref{dg_prop_2.1} is satisfied, and therefore $X$ coarsely embeds into a Hilbert space.
\end{proof}

\end{document}